\documentclass[reqno,12pt]{article}

\usepackage{a4wide}
\usepackage{amsmath,stmaryrd} 
\usepackage{amssymb}
\usepackage{amsthm}
\usepackage[utf8]{inputenc} 
\usepackage{graphicx} 
\usepackage{xcolor}
\usepackage[normalem]{ulem}

\numberwithin{equation}{section}

\newtheorem{theo}{Theorem}[section]
\newtheorem{conj}[theo]{Conjecture}
\newtheorem{coro}[theo]{Corollary}
\newtheorem{prop}[theo]{Proposition}

\newtheorem{lem}[theo]{Lemma}
\newtheorem{defi}[theo]{Definition}

\theoremstyle{remark}

\newtheorem*{Remark*}{Remark}
\newtheorem*{Remarks*}{Remarks}

\makeatletter
\newcommand*{\house}[1]{%
 \mathord{%
 \mathpalette\@house{#1}%
 }%
}
\newcommand*{\@house}[2]{%
 \dimen@=\fontdimen8 %
 \ifx#1\scriptscriptstyle\scriptscriptfont
 \else\ifx#1\scriptstyle\scriptfont
 \else\textfont\fi\fi
 3 %
 \sbox0{%
 $#1%
 \vrule width\dimen@\relax
 \overline{%
 \kern2\dimen@
 \begingroup 
 #2%
 \endgroup
 \kern2\dimen@
 }%
 \vrule width\dimen@\relax
 \mathsurround=1.5\dimen@ 
 $%
 }%
 \ht0=\dimexpr\ht0-\dimen@\relax
 \dp0=\dimexpr\dp0+2\dimen@\relax
 \vbox{%
 \kern\dimen@ 
 \copy0 %
 }%
}

\newcommand{\N}{\mathbb{N}}
\newcommand{\Z}{\mathbb{Z}}
\newcommand{\Q}{\mathbb{Q}}
\newcommand{\R}{\mathbb{R}}
\newcommand{\C}{\mathbb{C}}

\newcommand{\Qbar}{\overline{\mathbb Q}}

\newcommand{\etoile}{^*}

\newcommand{\calI}{\mathcal I}
\newcommand{\calE}{\mathcal E}
\newcommand{\calP}{\mathcal P}

\newcommand{\calV}{\mathcal V}

\newcommand{\Span}{{\rm Span}}

\begin{document}

\title{Zeros of $E$-functions and of exponential polynomials defined over $\Qbar$}
\date\today
\author{S. Fischler and T. Rivoal}
\maketitle

\begin{abstract}
Zeros of Bessel functions $J_\alpha$ play an important role in physics. They are a motivation for studying zeros of exponential polynomials defined over $\Qbar$, and more generally of  $E$-functions. In this paper we partially characterize $E$-functions with zeros of the same multiplicity, and prove a special case of a conjecture of Jossen on entire quotients of $E$-functions, related to Ritt's theorem and Shapiro's conjecture on exponential polynomials. We also deduce from Schanuel's conjecture many results on zeros of exponential polynomials over $\Qbar$, including $\pi$, logarithms of algebraic numbers, and zeros of $J_\alpha$ when $2\alpha$ is an odd integer. For the latter we define (if $\alpha\neq\pm1/2$) an analogue of the minimal polynomial and Galois conjugates of algebraic numbers. At last, we study conjectural generalizations to factorization and zeros of $E$-functions.
\end{abstract}

\section{Introduction}

 We recall the definition of $E$-functions. As usual, we embed $\Qbar$ into $\mathbb C$. 
A power series $f(x)=\sum_{n=0}^{\infty} a_n x^n/n! \in \Qbar[[x]]$ is said to be an $E$-function~if 

$(i)$ $f(x)$ is solution of a non-zero linear differential equation with coefficients in $\Qbar(x)$.
\smallskip

$(ii)$ There exists $C_1>0$ such that all Galois conjugates of $a_n$ have modulus $\le C_1^{n+1}$, for $n\ge 0$.

\smallskip

$(iii)$ There exists $C_2>0$ and a sequence of positive integers $d_n$, with $d_n \leq C_2^{n+1}$, such that $d_na_m$ are algebraic integers for all~$m\le n$.

\noindent   If $a_n\in \mathbb Q$, $(ii)$ and $(iii)$ read $\vert a_n\vert \le C_1^{n+1}$ and $d_na_m\in \mathbb Z$; in $(i)$, there exists such a differential equation with coefficients in $\mathbb Q(x)$, and the normalized one of minimal order also has coefficients in $\mathbb Q(x)$. The set of $E$-functions is a subring of $\Qbar[[x]]$, whose units are of the form $\alpha e^{\beta x}$ for some $\alpha, \beta \in \Qbar$, $\alpha\neq 0$. Abusing a standard terminology for $G$-functions, and because no confusion will be possible here, we shall say that an $E$-function is {\em globally bounded} if there exists an integer $D$ such that $D^{n+1}a_n$ is an algebraic integer for all $n\ge 0$ ({\em i.e.}, the associated $G$-function $\sum_{n\ge 0} a_n x^n$ is globally bounded in the usual sense). Finally, when some function $F$ is solution of a linear differential equation with polynomial coefficents, we shall often say that $F$ is holonomic.

\medskip

 Siegel~\cite{siegel} defined and studied $E$-functions in 1929, and this in particular enabled him to prove the Bourget hypothesis, {\em i.e.}, that $J_n$ and $J_m$ share no  common zero $\xi\in\C\etoile$, where $m$ and $n$ are distinct non-negative integers. Here $J_\alpha$ is one of the Bessel functions defined by $J_\alpha(x):=\sum_{m=0}^{\infty}\frac{(-1)^m}{m!\Gamma(m+\alpha+1)}(\frac{x}2)^{2m+\alpha}$. As this formula shows,  $\Gamma(\alpha+1)x^{-\alpha}J_\alpha(x)$ is an $E$-function for $\alpha\in\Q\setminus(-\N\etoile)$. 
Zeros of Bessel functions play an important role in many areas of physics; they are a motivation for studying zeros of $E$-functions.

\medskip

During an online talk in 2021, Jossen \cite{jossen} stated the following conjecture on the zeros of $E$-functions. 
\begin{conj}[Jossen] \label{conj2intro}
\begin{enumerate}
 \item[$(i)$] If $f$ and $g$ are two $E$-functions such that $f/g$ is an entire function, then $f/g$ is an $E$-function.
 \item[$(ii)$] If two $E$-functions $f$ and $g$ share at least one common root, then there exists a non-unit $E$-function $h$ such that $f/h$ and $g/h$ are $E$-functions.
\end{enumerate}
\end{conj}
His conjecture is inspired by Shapiro's conjecture \cite{shapiro} (see below) and by Ritt-type division theorems \cite{rahman, ritt} (see also \cite{shields}). Part $(i)$ holds if $f$ and $g$ are exponential polynomials over $\Qbar$   \cite{rahman, ritt}. Using \cite[Proposition~4.1]{beukers}, part $(i)$ holds if $g\in\Qbar[X]$ and part $(ii)$ does if the common root is algebraic. As far as we know, these are the only known results on Conjecture~\ref{conj2intro}.

\medskip

Using asymptotic expansions and the indicator function of a holomorphic function, we are able to prove the following.

\begin{theo}\label{propquoLgg} Let $g$ be an entire function such that $g^m$ is an $E$-function for some $m\geq 1$. Let $L\in\Qbar(x)[d/dx]$ be a differential operator such that ${L(g)}/g$ is an entire function. Then ${L(g)}/g\in\Qbar[x]$. 
\end{theo}

\begin{coro} 
Part $(i)$ of Jossen's conjecture holds true if $f$ is of the form $L(g)$ for some differential operator $L\in \Qbar(x)[d/dx]$, and in this case $L(g)/g$ is a polynomial.
\end{coro}

The present paper is devoted to a better understanding of zeros of $E$-functions, at least from a conjectural point of view. We conjecture that multiple zeros of an $E$-function $f$ always occur for a trivial reason: the zeros of multiplicity $j$ are exactly the zeros of an $E$-function $g_j$ such that $g_j^j$ divides $f$. A precise statement is the following.

\begin{conj} \label{conj1intro} Any    non-zero $E$-function $f$ can be written as 
$\prod_{j\in J} g_j(x)^j$ where  $J \subset \N\etoile$ is finite, and for each $j\in J$, $g_j$ is an $E$-function with zeros of multiplicity~$1$ such that $g_j$ and $g_{j'}$ have no common zero for $j\neq j'$. This representation is unique up to multiplication of the $g_j$'s by units. 
\end{conj}

If $f$ is not a unit, the set $J$ can be chosen as the set of multiplicities of zeros of $f$. 
As Jossen pointed out to us, it is then finite since except for singularities, 
all zeros of $f$ have multiplicity less than the order of a linear differential equation it satisfies.

In direction of this conjecture, we shall prove the following result using Theorem~\ref{propquoLgg} and the Hadamard factorization theorem. 

\begin{theo} \label{theo:1intro}
Let $f\in \mathbb Q[[x]]$ be a non-zero globally bounded $E$-function with zeros each of the same multiplicity $m\ge 2$. Assume that $f$ is solution of a non-trivial differential operator in $\mathbb Q(x)[d/dx]$ of order $\le m+1$.

Then there exists a non-zero globally bounded $E$-function $g$ with zeros of multiplicity~1 such that $f=g^m$, and $g(x)\in \alpha \mathbb Q[[x]]$ where $\alpha^m\in \Q^*$. Moreover $g$ is solution of a non-trivial differential~ope\-ra\-tor in $\mathbb Q(x)[d/dx]$ of order $\le 2$.
\end{theo}
The function $g$ is unique up to the multiplication of $\alpha$ by any $m$-th root of unity. The proof of Theorem~\ref{theo:1intro} shows more precisely that, under the same assumptions, if $f$ is of minimal (differential) order $m+1$ then $g$ is exactly of order $2$, while if $f$ is of order $\le m$ then both $g$ and $f$ are of order $1$, {\em i.e.} both of the form $p(x)e^{\xi x}$ with $\xi \in \mathbb Q$ and $p\in \mathbb Q[x]$. 

\bigskip

In the special case where the $E$-functions under consideration are exponential polynomials over $\Qbar$, we are going to prove that Conjectures \ref{conj2intro} and \ref{conj1intro}, and even much more precise statements, follow from Schanuel's conjecture. To stay in the framework of $E$-functions (and because it changes completely the situation), all exponential polynomials considered in the sequel will be over $\Qbar$, namely functions 
\begin{equation}
 \label{eqpolexpintro}f(x)=\sum_{i=1}^N P_i(x)e^{\beta_i x}
\end{equation} 
with $N\geq 1$, $P_1,\ldots,P_N\in\Qbar[X]$, $\beta_1,\ldots,\beta_N\in\Qbar$. We point out that we allow polynomial coefficients $P_i(x)$, instead of only constants in \cite{ritt} for instance. However, we {\em always} restrict ourselves to algebraic $\beta_i$ and polynomials $P_i$ with algebraic coefficients, instead of $\beta_i\in\C$ and $P_i\in\C[X]$ in the literature on exponential polynomials.

Units of the ring $\calP$ of exponential polynomials over $\Qbar$ are functions of the form $\lambda e^{\alpha x}$ with $\lambda,\alpha\in\Qbar$ and $\lambda\neq 0$. Irreducible elements are those non-units that cannot be written as a product of two non-units. Following Ritt, simple elements are those of the form $ e^{\alpha x} \sum_{i=1}^N \lambda_i e^{\beta r_i x}$
with $\alpha,\beta\in\Qbar$, $\beta\neq0$, $N\geq 2$, $\lambda_1,\ldots,\lambda_N\in\Qbar\etoile$, and pairwise distinct rational numbers $r_1,\ldots,r_N$. 
The {\em support} of such a simple function is the $1$-dimensional $\Q$-vector space spanned by $\beta$.

With these notations, any $f\in\calP$ can be written in a unique way (up to units) as a product of irreducible exponential polynomials  over $\Qbar$, and simple ones with pairwise distinct supports. This result is due to Ritt \cite{ritt1927} when the coefficients $P_i$ in Eq. \eqref{eqpolexpintro} are constant, and to MacColl \cite{maccoll} in the general setting (see also \cite{evdp}). This factorisation result enables one to define gcd's in the ring of exponential polynomials  over $\Qbar$ (see \cite[Theorem~3.1.18]{bg}). 

These results were first proved over $\C$, not $\Qbar$. Many papers are devoted to the study of common zeros of two exponential polynomials with complex coefficients (namely, $\beta_i\in\C$ and $P_i\in\C[X]$ in Eq. \eqref{eqpolexpintro}), for instance towards Shapiro's conjecture \cite{shapiro} (see also \cite{amt}): if two exponential polynomials (with constant coefficients, over $\C$) have infinitely many common zeros, then both are divisible by an exponential polynomial with infinitely many zeros. In this conjecture, one common zero is not enough: for instance, $e^x-e$ and $e^{x\sqrt2} -e^{\sqrt2} $ have a common zero at $x=1$, but no non-unit common factor since the gcd of these exponential polynomials is equal to 1. On the contrary, part $(ii)$ of Jossen's Conjecture~\ref{conj2intro} suggests that restricting to exponential polynomials over $\Qbar$ (which are $E$-functions) prevents this kind of behaviour: any common zero should be explained by a common factor. We prove this in Theorem~\ref{th10} below, assuming the following widely-believed very powerful conjecture, due to Schanuel \cite{lang}. 
\begin{conj}[Schanuel]
 \label{conjschanuelintro} Let $x_1,\ldots,x_n$ be complex numbers, linearly independent over $\Q$. Then $\Qbar(x_1,\ldots,x_n,e^{x_1}, \ldots, e^{x_n})$ has transcendence degree at least $n$ over $\Qbar$.
\end{conj}

\begin{theo}
 \label{th10} Assume Schanuel's conjecture holds. Let $f_1$, $f_2$ be non-zero exponential polynomials over $\Qbar$. Then $\frac{f_1}{\gcd(f_1,f_2)}$ and $\frac{f_2}{\gcd(f_1,f_2)}$ have no common zero in $\C\etoile$. 
 
 In other words, common zeros of $f_1$ and $f_2$ in $\C\etoile$ are exactly the zeros of $\gcd(f_1,f_2)$, and for any such $\xi$, the order of vanishing of $\gcd(f_1,f_2) $ at $\xi$ is the least of the orders of vanishing of $f_1$ and $f_2$ at $\xi$. \end{theo}

We point out that $0$ may be a common zero of $f_1$ and $f_2$ even in the case they are coprime, for instance if $f_1(x)=e^x-1$ and $f_2(x)=e^{x\sqrt2}-1$ (see the remark after Theorem~\ref{th3} in \S \ref{secpolexp}). 

Theorem~\ref{th10} enables us to define a kind of analogue of the minimal polynomial of an algebraic number for zeros of exponential polynomials over $\Qbar$. Let $\xi\in\C\etoile$ be a zero of  some $f\in\calP\setminus\{0\}$. Using Ritt's theorem, $\xi$ is a zero of either a simple or an irreducible exponential polynomial  over $\Qbar$. The former case means that $e^{\beta \xi}$ is algebraic for some $\beta\in\Qbar\etoile$; these numbers $\xi$ include $\pi$, logarithms of algebraic numbers, etc; they are periods. As noted by André \cite[\S\S 2.1 and 2.3]{andreGalois} (see also \cite[\S 3.3]{fresanx}), it is not clear how to define an analogue of the minimal polynomial for these numbers, in particular $\pi$. In this direction we prove the following.

\begin{theo}\label{theopi}
 Assume Schanuel's conjecture holds. Let $f$ be an exponential polynomial over $\Qbar$ such that $f(\pi)=0$. Then there exists $N\geq 1$ such that 
 $$\frac{f(x)}{e^{2i x/N}-e^{2i \pi/N}}$$
 is an exponential polynomial over $\Qbar$.
\end{theo}

In some sense, the family of functions $e^{2i x/N}-e^{2i \pi/N}$ would be an analogue of the minimal polynomial for $\pi$. Note that the function $e^{2i x/N}-e^{2i \pi/N}$ vanishes exactly at the points $(1+kN)\pi$, $k\in\Z$. Therefore $\pi$ is the only complex number at which all these functions vanish. This approach does not yield any definition for a ``conjugate'' of $\pi$ (recall that conjugates of an algebraic number are the complex roots of its minimal polynomial). 

The same holds, for instance, with $\log(2)$ and the functions $e^{x/N}-e^{\log(2)/N}$.

\medskip

If $\xi$ is a zero of a some $f\in\calP\setminus\{0\}$, and if $e^{\beta \xi}$ is algebraic for no $\beta\in\Qbar\etoile$, then $h(\xi)=0$ for some irreducible exponential polynomial $h$ over $\Qbar$. We suggest $h$ as an analogue of the minimal polynomial for $\xi$, and the zeros of $h$ as analogues of the conjugates of $\xi$, in view of the following result.

\begin{theo}\label{theoLambert}
 Assume Schanuel's conjecture holds. Let $\xi\in\C\etoile$, and $h$ be an irreducible exponential polynomial over $\Qbar$ such that $h(\xi)=0$. Then for any exponential polynomial $f$ over $\Qbar$ such that $f(\xi)=0$, $f/h$ is an exponential polynomial over $\Qbar$.
\end{theo}

For any integer $n\in\Z$, $\sqrt\pi x^{| n+1/2|} J_{n+1/2}(x)$ is an exponential polynomial  over $\Qbar$. If $n\in\{-1,0\}$, it is $\sqrt2 \cos x $ or $\sqrt2 \sin x$ and the situation is similar to that of Theorem~\ref{theopi}. Now for $n\in\Z\setminus \{-1,0\}$ this exponential polynomial over $\Qbar$ is irreducible (see the beginning of \S \ref{secpolexp}): Theorem~\ref{theoLambert} applies to all zeros $\xi\in\C\etoile$ of the Bessel function $J_{n+1/2}$. This suggests that $J_{n+1/2}$ can then be considered as an analogue of the minimal polynomial of $\xi$, and its zeros as analogues of its conjugates (with a slight abuse: to be precise, we mean $\sqrt\pi x^{-n-1/2} J_{n+1/2}(x)$ instead of $J_{n+1/2}$ itself).  We believe that $\xi$ is neither a period nor an exponential period. Accordingly this is probably the first attempt to define these notions for $\xi$. We would like to point out also that it would be very interesting to have a transitive group action on the set of conjugates of $\xi$, as the Galois action in the algebraic setting. The situation is the same for values $\xi=W(c)$ of the Lambert $W$ function at non-zero algebraic points $c$ (see the beginning of \S \ref{secpolexp}).

\medskip

In order to deal with zeros of  Bessel functions $J_\alpha$ of any rational order $\alpha$, we propose conjectures to extend the above properties to  the ring $\calE$ of $E$-functions. We start with a generalization of Ritt's theorem, that enables us to define the gcd of two $E$-functions so that common zeros of $E$-functions should be the zeros of their gcd (as in Theorem~\ref{th10}). Then we adapt Theorems \ref{theopi} and \ref{theoLambert}. For Bessel functions, we deduce from our general conjectures the following one (recall that $J_{\pm 1/2}(x)$ is a simple exponential polynomial over~$\Qbar$).

\begin{conj}\label{conjJ} Let $\alpha\in\Q\setminus(\{\pm 1/2\}\cup(-\N\etoile)) $, and $f$ be an $E$-function such that $f$ and $J_\alpha$ share a common zero in   $\C\etoile$. Then $f(x)= g(x) \Gamma(\alpha+1)x^{-\alpha}J_\alpha(x)$ for some $E$-function~$g$. 
\end{conj}

This conjecture suggests $\Gamma(\alpha+1)x^{-\alpha}J_\alpha(x)$  and its zeros as analogues of the minimal polynomial and the Galois conjugates of $\xi$, when $\xi $ is a zero of $J_\alpha$ (as above when $\alpha-1/2\in\Z$). 

\bigskip

The structure of this paper is as follows. In \S \ref{ssec:auxi}, we prove Theorem~\ref{propquoLgg}, which implies  a special case of Jossen's conjecture and will be used in \S \ref{secpreuvetheo1} to prove Theorem~\ref{theo:1intro}. Then we move in \S \ref{secpolexp} to zeros of exponential polynomials over $\Qbar$, and conclude in \S\S \ref{secstruct} and \ref{seczerosE} with factorization and zeros of $E$-functions.

\section{A special case of Jossen's conjecture} \label{ssec:auxi}

In this section, we prove Theorem~\ref{propquoLgg} stated in the introduction. We shall use the indicator function $h$ associated with an holomorphic function $f$ of exponential type in a sector $\alpha\leq \arg(x) \leq \beta $. We recall that $f$ is said to be of exponential type in such a sector if there exists $\tau>0$ such that, for any $x\in\C\etoile$ with $\alpha\leq \arg (x) \leq \beta $, we have $|f(x)|\leq \exp(\tau |x|)$. Then for any $\theta\in[\alpha,\beta]$, $h(\theta)$ is defined by 
$$ h(\theta) = \limsup_{r\to+\infty} \frac{\log | f(re^{i\theta})|}{r}.$$
We refer to \cite[Ch. 5]{boas} for the general properties of the indicator function. We shall use the following consequence of \cite[Theorem~6.2.4, Ch.~6, p.~82]{boas}:

\begin{prop}\label{theoboas}
 If $f$ is holomorphic and of exponential type in the (closed) upper half-plane, bounded on $\R$ and such that $h(\pi/2)\leq 0$, then it is bounded on the upper half-plane.
\end{prop}

\begin{proof}[Proof of Theorem~\ref{propquoLgg}] For any $\theta\in[-\pi, \pi]$ outside a finite set, the $E$-function $g^m$ has an asymptotic expansion in a large sector bisected by $\theta$ of the form $\sum_{\rho\in\Sigma} f_\rho(z)e^{\rho z}$ with $f_\rho(x)\in{\rm NGA}\{1/x\}_1$ (see \cite[Théorème de dualité]{andre} and \cite[\S 4.1]{ateo}). Up to changing $z$ to $e^{i\alpha}x$ for a suitable $\alpha$, we may assume that such an expansion holds in a large sector bisected by $\pi/2$. Shrinking $\Sigma$ if necessary, we also assume that $f_\rho\neq 0$ for any $\rho\in\Sigma$. 

Now we fix $\theta\in[0,\pi]$ and consider $g^m(x)$ as $\vert x\vert\to+\infty$ with $\arg(x) = \theta$; notice that in this direction, $\vert\exp(\rho x)\vert = \exp(\vert x\vert{\rm Re}(\rho e^{i\theta}))$ for any $\rho\in\Sigma$. Except for finitely many values of $\theta$ (related to Stokes' phenomenon), the maximum of ${\rm Re}(\rho e^{i\theta})$ as $\rho$ ranges through $\Sigma$ is obtained for only one value $\rho_\theta$. Then, as $\vert x\vert\to+\infty$ in this direction, $\exp(\rho x)$ is exponentially smaller than $\exp(\rho_\theta x)$, for any $\rho\in\Sigma\setminus\{\rho_\theta\}$. Consequently, there exist $a_\theta\in\Q$, $j_\theta\in\N$, and $c_\theta\in\C\etoile$ such that $$g^m(x)\sim c_\theta x^{a_\theta} (\log x)^{j_\theta} e^{\rho_\theta x}$$ as $\vert x\vert\to+\infty$ with $\arg(x)= \theta$. Taking $m$-th roots provides $d_\theta \in\C\etoile$ such that $$g (x)\sim d_\theta x^{a_\theta/m} (\log x)^{j_\theta/m} e^{\rho_\theta x/m}.$$ Moreover, the same argument shows that either $(Lg)(x)\sim d_\theta x^{a'_\theta} (\log x)^{j'_\theta} e^{\rho_\theta x/m}$ for some $a'_\theta, j'_\theta\in\Q$, and $d_\theta\in\C\etoile$, or $(Lg)(x)=o( e^{\rho_\theta x/m})$. The latter happens if the part relative to $e^{\rho_\theta x/m}$ in the asymptotic expansion of $g$ is annihilated by $L$. In both cases, we obtain $\frac {L(g)}g(x) = \mathcal{O}(x^{A_\theta})$ for some $A_\theta\in\N$, as $ \vert x\vert\to+\infty$ with $\arg(x) = \theta$.

Changing again $x$ to $e^{i\alpha }x$ if necessary, we may assume that for any $\theta\in\{0, \pi/2, \pi\}$ the above-mentioned maximum of ${\rm Re}(\rho e^{i\theta})$ is obtained for only one value $\rho_\theta$. Then we have $\frac { Lg}g(x) = \mathcal{O}(x^{A_\theta})$ as $\vert x\vert\to+\infty$ with $\arg(x) = \theta$. Now let $A=\max(A_0,A_\pi)$ and $f(x) = (x+i)^{-A} \frac {L(g)}g(x) $. Then $f$ is holomorphic and of exponential type in the closed upper half-plane, bounded on $\R$. Its indicator function satisfies $h(\pi/2)\leq 0$. Using Proposition~\ref{theoboas}, it is bounded, so that $\vert\frac {L(g)}g(x) \vert\leq \vert x+i\vert^A $ for any $x\in\C$ with ${\rm Im}(x)\geq 0$. The same proof yields $\vert\frac { L(g)}g(x) \vert\leq \vert x-i\vert^A $ when ${\rm Im}(x)\leq 0$. Therefore $L(g)/g$ has (at most) a polynomial  growth at infinity: it is a polynomial by Liouville's theorem. Now $g^m$ has algebraic Taylor coefficients at 0: so do $g$, $L(g)$ and $L(g)/g$. Finally, $L(g)/g\in\Qbar[x]$: this concludes the proof of Theorem~\ref{propquoLgg}. \end{proof}

\section{Proof of Theorem~\ref{theo:1intro}} \label{secpreuvetheo1}

We do not repeat the assumptions of Theorem~\ref{theo:1intro} which are assumed throughout this section.
If $0$ is a zero of $f$ (of multiplicity $m$), $\widetilde{f}(x):=f(x)/x^m\neq 0$ is still an $E$-function in $\mathbb Q[[x]]$ with zeros all of multiplicity $m$, all different from 0, and $\widetilde{f}(0)\in \mathbb Q^*$. It is then clearly enough to prove the theorem for $\widetilde{f}(x)/\widetilde{f}(0)$, which is globally bounded and solution of a differential operator in $\Qbar(x)[d/dx]$ of order $\le m+1$. Therefore, without loss of generality, we assume that $f(0)=1$ from now on.

\medskip

Before going further and because this will appear below, we recall that $E$-functions have been defined by Siegel~\cite{siegel} in a more general way, {\em i.e.}, the two bounds $(\cdots)\le C_i^{n+1}$ in the definition in the introduction of $E$-functions (which are often said to be {\em in the strict sense}) are replaced by: for all $\varepsilon>0$, $(\cdots)\le n!^{\varepsilon}$ for all $n\ge N(\varepsilon)$. A globally bounded $E$-function as defined in the introduction is automatically a strict $E$-function by a theorem of Perron \cite{perron}; see the details in \cite[p. 715]{andre}. We have decided to state our conjectures for $E$-functions in the strict sense, but we could formulate the same conjectures for $E$-functions in Siegel's sense {\em mutatis mutandis}. However, it is believed that an $E$-function in Siegel's sense is automatically a strict $E$-function; see again \cite[p.~715]{andre} for a discussion. 

\subsection{Existence of $g$ entire of order $\le 1$ such that $f=g^m$ and $g(0)=1$}

We recall that the order of an entire function $h$ is the infimum of the set of all $c$ such that $h(x)=O(\exp(|x|^c))$ as $x$ tends to infinity in the complex plane.
Since $f$ is of order $\rho(f)\le 1$, by the Hadamard factorization theorem \cite[Chapter~2]{boas}, we have
$$
f(x)=e^{\beta x} \prod_{\zeta \in Z(f)} \Big(1-\frac{x}{\zeta}\Big)^me^{mx/\zeta},
$$
with $\beta=f'(0)\in \mathbb Q$, where $Z(f)$ is the set of zeros of $f$.
Then 
$$
g(x):=e^{\beta x/m} \prod_{\zeta \in Z(f)} \Big(1-\frac{x}{\zeta}\Big)e^{x/\zeta}
$$
is also an entire function of order $\rho(g) \le 1$, and such that $g(0)=1$ and $g^m=f$.

Let 
$$
g(x)=\sum_{n=0}^\infty \frac{b_n}{n!}x^n, \quad b_0=1.
$$
By \cite[p.~9]{boas}, we have 
$$
\limsup_{n\to +\infty} \frac{n\log(n)}{\log(n!/b_n)}=\rho(g)\le 1,
$$
so that for all $\varepsilon>0$, we have $b_n=\mathcal{O}(n!^{\varepsilon})$ for all $n\ge N(\varepsilon)$. 

\medskip

This is the growth requested on the sequence $(b_n)_n$ for $g$ to be an $E$-function in Siegel's sense. Moreover, if $g$ can be proved to be holonomic, then by above mentioned theorem of Perron \cite{perron} this bound automatically implies that $b_n=\mathcal{O}(C^n)$ for some $C>0$ which is the growth requested in $(ii)$ for $g$ to be an $E$-function in the strict sense. 

\subsection{Denominators of the Taylor coefficients of $g$ at the origin}

Let $f(x)=\sum_{n=0}^\infty \frac{a_n}{n!} x^n$ with $a_0=1$. We consider again the function $g$ defined in the previous section. Locally around $x=0$, we have (because $g(0)=1$): 
\begin{align*}
g(x)&=f(x)^{1/m}=\Big(1+\sum_{n=1}^\infty \frac{a_n}{n!}x^n\Big)^{1/m}
\\
&=\sum_{k=0}^\infty \binom{1/m}{k}
\Big(\sum_{n=1}^\infty \frac{a_n}{n!}x^n\Big)^k
\\
&=\sum_{\ell =0}^\infty \frac{b_\ell}{\ell!}x^\ell,
\end{align*}
so that for any $\ell\ge 1$,
\begin{equation} \label{eq:expressionbell}
b_\ell
=\sum_{k=1}^\ell 
(-1)^{k}\frac{(-1/m)_k}{k!}
\sum_{\underset{n_j\ge 1}{n_1+\cdots+n_k=\ell}} \binom{\ell}{n_1, \ldots, n_k} a_{n_1}\cdots a_{n_k} \in \mathbb Q.
\end{equation}
Since $f$ is assumed to be globally bounded and $a_0=1$, consider an integer $D\ge 1$ such that $D^{n}a_n\in \mathbb Z$ for all $n\ge 1$. It is well-known that $m^{2k}\frac{(-1/m)_k}{k!}\in \mathbb Z$ for all $k\ge 1$ (see for instance \cite[Theorem~4]{drr}). Therefore, we deduce from \eqref{eq:expressionbell} that $(m^{2}D)^{\ell} b_\ell \in \mathbb Z$ for all $\ell\ge 0$. This estimate is what is requested in $(iii)$ for $g$ to be an $E$-function, and even a globally bounded one.

However, if $f$ is not assumed to be globally bounded, it does not seem possible to deduce from \eqref{eq:expressionbell} that the least positive common denominator of $b_0, \ldots, b_n$ is bounded by $n!^{\varepsilon}$, resp by $C^{n+1}$, if the least positive common denominator of $a_0, \ldots, a_n$ is bounded by $n!^{\varepsilon}$, resp by $D^{n+1}$.

\medskip

In summary, the results proven so far show that if $g$ is also holonomic, then it is a globally bounded $E$-function.

\subsection{Holonomicity of $g$}

Let $m\ge 2$. By Leibniz formula, we have for all $k\ge 0$: 
$$
f^{(k)}=(g^m)^{(k)} = \sum_{\ell_1+\cdots+\ell_m=k} \binom{k}{\ell_1,\ldots, \ell_m} g^{(\ell_1)}\cdots g^{(\ell_m)}
$$
where the sum runs over the integers $\ell_j\ge 0$ such that $\ell_1+\cdots+\ell_m=k$.

We claim that there exist $P_{k,m}\in \mathbb Z[X_1, \ldots, X_{k+1}]\setminus\{0\}$ and $c_{k,m}\in \mathbb N^*$ such that for $k\le m-1$, 
$$
(g^m)^{(k)} = gP_{k,m}(g, g', \ldots, g^{(k)})
$$
and for $k=m$ and $k=m+1$, 
$$
(g^m)^{(k)} = gP_{k,m}(g, g', \ldots, g^{(k)}) + c_{k,m} g^{(k-m+1)}(g')^{m-1}.
$$ 

Let $k\le m-1$. We see that in Leibniz formula, given an $m$-tuple $(\ell_1, \ldots, \ell_m)$ such $\ell_1+\cdots+\ell_m=k$, we obviously cannot have $\ell_j\ge 1$ for each $j$, hence $g$ always appears in the products $g^{(\ell_1)}\cdots g^{(\ell_m)}$ and the claim follows in this case. 

Let $k=m$. If one of the $\ell_j=0$, the corresponding term $g^{(\ell_1)}\cdots g^{(\ell_m)}$ contributes to $gP_{k,m}(g, g', \ldots, g^{(k)})$. If none of the $\ell_j=0$, {\em i.e.}, all are $\ge 1$, then in fact $\ell_1= \cdots= \ell_m=1$ because otherwise $\ell_1+\cdots+\ell_m\ge m+1>k$; hence in that case, we have a unique term $m! \cdot (g')^m$. This proves the claim in this case.

Let $k=m+1$. Again if one of the $\ell_j=0$, the corresponding term $g^{(\ell_1)}\cdots g^{(\ell_m)}$ contributes to $gP_{k,m}(g, g', \ldots, g^{(k)})$.
If none of the $\ell_j=0$, {\em i.e.}, all are $\ge 1$, then exactly one must be equal to 2 and the others must be equal to 1. Indeed, if at least two are $\ge 2$, then $\ell_1+\cdots+ \ell_m\ge m+2>k$ while if they are all equal to 1, we have $\ell_1+\cdots+ \ell_m= m<k$; hence in that case we have a (non-empty) sum of terms all of the form $c\cdot g''(g')^{m-1}$, $c\in \mathbb N^*$. This proves the claim in this case as well.

\medskip

Let us now assume that $f$ is solution of a differential operator in $\mathbb Q(x)[d/dx]$ of order $m+1$, {\em i.e.}, we have 
$$
\sum_{k=0}^{m+1} a_k(x)f^{(k)}(x)=0, \quad a_{m+1}\neq 0,\, a_k\in \mathbb Z[x].
$$
Using the above expressions for $(g^m)^{(k)}$, we then obtain an algebraic differential equation for $g$ of the form
$$
\big(c_{m+1,m}a_{m+1}g''+c_{m,m}a_mg'\big) g'^{m-1}=gQ\big(x,g,g', \ldots, g^{(m+1)}\big), \quad Q\in \mathbb Z[X_0, X_1, \ldots X_{m+2}]. 
$$
Since $g$ has simple zeros, it has no common zeros with $g'$. Since $Q(x,g,g', \ldots, g^{(m+1)})$ is an entire function, all these simple zeros are zeros of $c_{m+1,m}a_{m+1}g''+c_{m,m}a_mg'$ so that 
\begin{equation}\label{eq:defh}
h:= \frac{c_{m+1,m}a_{m+1}g''+c_{m,m}a_mg'}{g} 
\end{equation}
is an entire function. Hence by Theorem~\ref{propquoLgg}, 
$h$ is in $\Qbar[x]$. Moreover, as the right-hand side of \eqref{eq:defh} is in 
$\mathbb Q((x))$, we deduce that $h\in \mathbb Q[x]$. The equation
$$
c_{m+1,m}a_{m+1}g''+c_{m,m}a_mg'-hg=0
$$
then shows that $g$ is solution of a differential operator in $\mathbb Q(x)[d/dx]$ of order 2. Note that if $f$ is of minimal differential order $m+1$, then $g$ is of minimal differential order $2$: indeed, by what precedes, $g$ would otherwise be of minimal order $1$ and thus $f$ as well, which is not possible.

\medskip

Finally, let us assume that $f$ is solution of a differential operator in $\mathbb Q(x)[d/dx]$ of order~$\le m$. We distinguish two cases.

First case. Let us assume that $f$ has infinitely many zeros of multiplicity $m$. Then at one such zero $\xi$ which is not one of the finitely many singularities of the differential equation of $f$, we have $f(\xi)=f'(\xi)=\cdots=f^{(m-1)}(\xi)=0$. Since $\xi$ is an ordinary point of this equation of order $\le m$, by the Cauchy-Lipschitz theorem, $f$ must be identically zero, which is excluded. 

Second case. Let us assume that $f$ has finitely many zeros. Then $f\in \mathbb Q[[x]]$ is of the form $p(x)e^{\xi x}$ by the Hadamard factorization formula, with $\xi\in \mathbb C$ and $p\in\mathbb C[x]$. Since $f$ is an $E$-function, we know that $\xi\in \Qbar$ hence $p\in \Qbar[x]$. If $\xi=0$, then $p=f$ is in $\mathbb Q[x]$. Let us now assume that $f$ is transcendental, {\em i.e.}, $\xi \neq 0$. By \cite[Theorem~3.4]{brs}, there exist unique $u,v\in \mathbb \Qbar[x]$ and $h\in \Qbar[[x]]$ a purely transcendental $E$-function~(\footnote{By definition, a {\em purely transcendental} $E$-function $h$ is such that $h(\eta)\notin \Qbar$ for all $\eta\in \Qbar^*$.}) such that $v$ is monic, $v(0)\neq 0$, $\deg(u)<\deg(v)$ and $f=u+vh$ (canonical decomposition of $f$); moreover $u,v,h$ have rational Taylor coefficients because $f\in \mathbb Q[[x]]$. Let $p_\mu\in \Qbar^*$ denote the leading coefficient of $p$: since $p_\mu e^{\xi x}$ is purely transcendental, $(p(x)/p_\mu)\cdot p_\mu e^{\xi x}$ is the canonical decomposition of $f$ and thus $(p(x)/p_\mu)\in \mathbb Q[x]$ and $p_\mu e^{\xi x} \in \mathbb Q[[x]]$. Hence $p_\mu \in \mathbb Q$ and consequently $f\in \mathbb Q[x]e^{\mathbb Qx}$ with roots of multiplicity $m$. Therefore $g$ is in $\mathbb Q[x]e^{\mathbb Qx}$, is solution of a differential operator in $\mathbb Q(x)[d/dx]$ of order $1$, and is globally bounded because this is the case of all $E$-functions in $\Qbar[x]e^{\Qbar x}$. This case is thus possible.

\section{Zeros of exponential polynomials  over $\Qbar$} \label{secpolexp}

In this section we consider exponential polynomials over $\Qbar$, defined as functions 
\begin{equation}
 \label{eqpolexp}f(x)=\sum_{i=1}^N P_i(x)e^{\beta_i x}
\end{equation} 
with $N\geq 1$, $P_1,\ldots,P_N\in\Qbar[X]$, $\beta_1,\ldots,\beta_N\in\Qbar$. We emphasize that we allow polynomial coefficients $P_i(x)$, instead of only constants in \cite{ritt} for instance. However we restrict to algebraic $\beta_i$ and polynomials with algebraic coefficients, instead of $\beta_i\in\C$ and $P_i\in\C[X]$ in the literature on exponential polynomials. This changes the situation completely. This restriction is necessary for exponential polynomials to be $E$-functions. {\em All exponential polynomials considered in this section are over $\Qbar$, unless stated otherwise.}

Units and irreducible elements of the ring $\calP$ of exponential polynomials have been defined in the introduction. For instance, given $c\in\Qbar\etoile$, the function $h_c(x) = x e^x - c$ is irreducible in $\calP$. Indeed, using Ritt's theory this follows from the irreducibility of $X_0 X_1^n - c$ in $\Qbar[X_0,X_1]$ for any $n\geq 1$ (which is a consequence of Eisenstein's criterion applied to $ X_1^n - c/X_0$, seen as a polynomial in $X_1$ with coefficients in the factorial ring $\Qbar[1/X_0]$). 

For any $n\in\Z$, we have $\sqrt{\pi/2}x^{|n+1/2|}J_{n+1/2}(x)=e^{-ix}(A_n(x)e^{2ix}+B_n(x))$ with $A_n,B_n\in\Q[X]$. Ritt's theory shows that the irreducibility of this exponential polynomial in $\calP$ follows from that of $A_n(X)Y^k+B_n(X)$   in $\Qbar[X,Y]$ for any $k\geq 1$; we shall prove it now for any $n\in\Z\setminus\{-1,0\}$ using Eisenstein's criterion. To begin with, we notice that $x^{-|n+1/2|}A_n(x)e^{ix}$ and $x^{-|n+1/2|}B_n(x)e^{-ix}$ are solution of the Bessel differential operator of order 2 that annihilates $J_{n+1/2}$, of which $0$ is the only finite singularity. Therefore $A_n(x) $ and $B_n(x) $ are also solutions of differential equations of order 2 with no non-zero finite singularities. The Cauchy-Lipschitz theorem shows that they have only simple roots (except possibly $0$). Since $n\not\in \{-1,0\}$,  at least one of $A_n$ and $B_n $ has a simple root $\xi$. It cannot be a common root of $A_n$ and $B_n $, since otherwise $J_{n+1/2}$ would have a non-zero algebraic zero, 
in contradiction with Siegel's theorem. Therefore Eisenstein's criterion applies with the irreducible polynomial $X-\xi$, and proves that $A_n(X)Y^k+B_n(X)$ is irreducible in $\Qbar(X)[Y]$ and then in $\Qbar[X,Y]$ since $A_n$ and $B_n$ are coprime. This concludes the proof that 
 $\sqrt{\pi/2}x^{|n+1/2|}J_{n+1/2}(x)$ is  irreducible    in $\calP$.

\medskip

Theorem~\ref{th10} stated in the introduction is an immediate consequence of the following result, since units of $\calP$ have no zeros.

\begin{theo}
 \label{th3} Assume that Schanuel's conjecture holds. Let $f_1,f_2$ be non-zero exponential polynomials with (at least) a common zero $\xi\in\C\etoile$. Then there exists $f\in\calP\setminus\{0\}$, which vanishes at $\xi$, such that $f_1/f$ and $f_2/f$ are exponential polynomials. 
\end{theo}

We point out that Theorem~\ref{th3} would be false with $\xi=0$; for instance $f_1(x)=e^x-1$ and $f_2(x)=e^{x\sqrt2}-1$ both vanish at $0$, but are not multiple of an exponential polynomial vanishing at $0$ because they are simple with distinct supports (see Lemma~\ref{lemmultisple} below). Another remark is that Theorem~\ref{th3} is remniscent of the Shapiro conjecture: if $f_1,f_2$ are exponential polynomials with constant complex coefficients ({\em i.e.}, of the form \eqref{eqpolexp} with $\beta_i,P_i\in\C$) with infinitely many common zeros, then the conclusion of Theorem~\ref{th3} holds. In the complex setting, one common zero is not enough: for instance $e^x-e$ and $e^{x\sqrt2}-e^{\sqrt2}$ have a common zero at $1$, but are not multiple of an exponential polynomial vanishing at~$1$. This example shows that restricting to 
$\beta_i\in\Qbar$ and $P_i\in\Qbar[X]$ in Eq. \eqref{eqpolexp} changes completely the situation.

\bigskip

\begin{coro}
 \label{cormultizeros2} If Schanuel's conjecture holds then irreducible exponential polynomials have only simple zeros.
\end{coro}

\begin{proof}[Proof of Corollary~\ref{cormultizeros2}] Let $h\in\calP$ be irreducible. If $\xi$ is a multiple zero of $h$, then it is also a zero of $h'$. Using Theorem~\ref{th3} it is a zero of $\gcd(h,h')$ so this $\gcd$ is not a unit. Since $h$ is irreducible, this $\gcd$ is equal to $h$ (up to a unit), and $h$ divides $h'$ in $\calP$. Denoting by $\omega$ the multiplicity of $\xi$ as a zero of $h$, we obtain $\omega\leq \omega-1$ since all exponential polynomials are holomorphic at $\xi$: this is a contradiction, and Corollary~\ref{cormultizeros2} is proved.
\end{proof}

\bigskip

Let us deduce the following consequence of Theorem~\ref{th3}; it contains Theorems \ref{theopi} and \ref{theoLambert} stated in the introduction. 

\begin{coro}
 \label{cordicho} Assume that Schanuel's conjecture holds, and let $\xi\in\C\etoile$ be a zero of an exponential polynomial. Then 
 one, and only one, of the following holds:
 \begin{itemize}
 \item We have $e^{\beta\xi}\in\Qbar$ for some $\beta\in\Qbar\etoile$, and for any $f\in\calP$ with $f(\xi)=0$ there exists $N\geq 1$ such that $f$ is divisible by $e^{\beta x/N}-e^{\beta\xi/N}$ in $\calP$.
 \item We have $h(\xi)=0$ for some irreducible $h\in\calP$, and $h$ divides in $\calP$ any exponential polynomial that vanishes at $\xi$.
 \end{itemize}
\end{coro}

\begin{proof}[Proof of Corollary~\ref{cordicho}] To begin with, let $\xi\in\C\etoile$ and $h,f\in\calP$ be such that $h(\xi)=f(\xi)=0$, with $h$ irreducible. Using Theorem~\ref{th10}, $\gcd(h,f)$ vanishes at $\xi$ so it is not a unit. Since $h$ is irreducible, it is $h$ (up to a unit) so that $h$ divides $f$ in $\calP$. This proves the second part. Moreover, if $f(x)=e^{\beta x}-c$ with $\beta,c\in\Qbar\etoile$ then $f$ is simple so $h$ cannot divide $f$ (by uniqueness in Ritt's factorization theorem): $\xi$ cannot be in both situations of Corollary~\ref{cordicho}. 

At last, if $\xi\in\C\etoile$ is a zero of some $f_0\in\calP\setminus\{0\}$ but of no irreducible $h\in\calP$, then $\xi$ a zero of a simple factor $g$ of $f_0$. Up to a unit, we can write $g(x)=P(e^{\beta x})$ for some $\beta\in\Qbar\etoile$ and $P\in\Qbar[X]\setminus\{0\}$. Therefore $e^{\beta \xi}$ is a zero of $P$: it is algebraic. If $f\in\calP$ is such that $f(\xi)=0$, there exists $\beta_1\in\Qbar\etoile$ and $P_1\in\Qbar[X]\setminus\{0\}$ such that 
 $g_1(x)=P_1(e^{\beta_1 x})$ divides $f$ in $\calP$ and vanishes at $\xi$. Both $e^{\beta \xi}$ and $e^{\beta_1 \xi } = ( e^{\beta \xi}) ^{\beta_1/\beta}$ are algebraic, and $ e^{\beta \xi}\neq 1$: the Gel'fond-Schneider theorem yields $ \beta_1/\beta\in\Q$. This provides $\beta_2\in\Qbar\etoile$ such that $\beta,\beta_1\in\beta_2\Z$. Up to multiplying $g$ and $g_1$ with suitable units, we may assume that $\beta=n\beta_2$ and $\beta_1=n_1\beta_2$ with $n,n_1\in\N\etoile$. Letting $Q(X)=P(X^n)$ and $Q_1(X)=P_1(X^{n_1})$ we have $g(x)=Q( e^{\beta_2x} )$ and $g_1(x)=Q_1( e^{\beta_2x} )$. Since $g(\xi)=g_1(\xi)=0$, the polynomials $Q$ and $Q_1$ have $ e^{\beta_2\xi}$ as a common root, so they are multiples of $X- e^{\beta_2\xi}$ in $\Qbar[X]$. Finally the simple function $ e^{\beta_2x}- e^{\beta_2\xi}$, with $\beta_2=\beta/n$, vanishes at $\xi$ and divides $f$. This concludes the proof of Corollary~\ref{cordicho}.
 \end{proof}

 \bigskip

\begin{proof}[Proof of Theorem~\ref{th3}] We denote by $\beta_1, \ldots, \beta_N$ the exponents in an expression \eqref{eqpolexp} of $f_1$, and by $\beta'_1, \ldots, \beta'_{N'}$ those for $f_2$. Let $W$ be the $\Z$-module generated by $\beta_1, \ldots, \beta_N, \beta'_1, \ldots, \beta'_{N'}$. There exists a $\Z$-basis $\alpha_1,\ldots,\alpha_p$ of $W$; all exponents $\beta_i,\beta'_j$ can be written as $\Z$-linear combinations of $\alpha_1,\ldots,\alpha_p$. Multiplying $f_1$ and $f_2$ with $e^{\gamma x}$ for a suitable $\gamma\in W\subset \Qbar$, we may assume that these linear combinations involve only non-negative coefficients. Then we have
$$
f_1(x) = P_1(x, e^{\alpha_1 x}, \ldots, e^{\alpha_p x}) \quad \mbox{and} \quad 
f_2(x) = P_2(x, e^{\alpha_1 x}, \ldots, e^{\alpha_p x})
$$
for some $P_1,P_2\in\Qbar[X_0,\ldots,X_p]\setminus\{0\}$. The complex numbers $\alpha_1\xi,\ldots,\alpha_p\xi$ are linearly independent over $\Q$ because $\alpha_1,\ldots,\alpha_p$ are, and $\xi\neq 0$; but they are linearly dependent over $\Qbar$. Schanuel's conjecture implies that 
$$
\Qbar(\alpha_1\xi,\ldots,\alpha_p\xi,e^{\alpha_1 \xi}, \ldots, e^{\alpha_p \xi}) = \Qbar( \xi,e^{\alpha_1 \xi}, \ldots, e^{\alpha_p \xi})
$$
has transcendence degree at least $p$ over $\Qbar$. In other words, letting 
$$
J = \{S\in \Qbar[X_0,\ldots,X_p], \quad S(\xi, e^{\alpha_1 \xi}, \ldots, e^{\alpha_p \xi})=0\},
$$
the zero set of $J$ in $\Qbar^{p+1}$ has dimension at least $p$; since $P_1$ and $P_2$ are non-zero and belong to $J$, this dimension is equal to $p$. Therefore $J$ is principal: there exists $P\in \Qbar[X_0,\ldots,X_p]\setminus\{0\}$ such that $J$ consists in all multiples of $P$, and accordingly there exist $T_1,T_2\in \Qbar[X_0,\ldots,X_p]$ such that $P_1=T_1P$ and $P_2=T_2P$. Letting 
$$
f(x) = P(x, e^{\alpha_1 x}, \ldots, e^{\alpha_p x}), \quad
\widetilde f_1(x) = T_1(x, e^{\alpha_1 x}, \ldots, e^{\alpha_p x}), \quad
\widetilde f_2(x) = T_2(x, e^{\alpha_1 x}, \ldots, e^{\alpha_p x}),
$$
we have $f(\xi)=0$, $f_1=f \widetilde f_1$ and $f_2=f \widetilde f_2$. This concludes the proof of Theorem~\ref{th3}. 
\end{proof}

\bigskip

To conclude this section we state the following well-known result, valid also in the complex setting. It follows 
immediately from \cite[Proposition~3.1.1]{bg}.

\begin{lem}
 \label{lemmultisple}
 Let $f_1,f_2$ be non-zero exponential polynomials such that $f_1f_2$ is simple. Then $f_1$ and $f_2$ are simple and $f_1$, $f_2$, $f_1f_2$ have the same support.
\end{lem}

This lemma shows that if $f$ is simple, then all divisors of $f$ in $\calP$ are also simple with the same support.

\section{Factorization  of $E$-functions} \label{secstruct}

\newcommand{\supp}{{\rm supp}}
\newcommand{\Ecroix}{\calE^\times}
\newcommand{\oun}{{\rm ord}_1}

In this section we describe the (conjectural) structure of the ring $\calE$ of $E$-functions.

A {\em unit} $u\in\calE$ is an $E$-function such that $uv=1$ for some $v\in\calE$; we denote by $\Ecroix$ the set of units. Units are exactly non-vanishing $E$-functions, {\em i.e.}, the functions of the form $\lambda e^{\alpha x}$ with $\lambda,\alpha\in\Qbar$ and $\lambda\neq 0$ (using the Hadamard factorization theorem). 

An {\em irreducible} element is a non-unit $f\in\calE$ such that $f=gh$ with $g,h\in\calE$ implies that either $g$ or $h$ is a unit. We shall see in the proof of Proposition~\ref{propordreannul} below that if $f(x_0)=0$ for some $x_0\in\Qbar$, then $f$ is irreducible if and only if $f(x)=(x-x_0)u(x)$ for some unit $u$. In particular a polynomial $f\in\Qbar[x]$ is irreducible in $\calE$ if, and only if, it has degree one. We conjecture Bessel's $E$-functions $\Gamma(\alpha+1) x^{-\alpha} J_\alpha(x)$ to be irreducible when $\alpha\in\Q\setminus(\{\pm1/2\}\cup(-\N\etoile))$.  

The integral domain $\calE$ is not a factorial ring. Indeed, for any $N\geq 1$ we have 
\begin{equation}
 \label{eqnonfact} e^x-1 = \prod_{k=1}^N \Big( e^{x/N} - e^{2i\pi k/N}\Big)
\end{equation}
whereas in a factorial ring, no non-zero element can be written as a product of arbitrarily many non-units. 
In the spirit of \cite[Theorem~5]{factoreop} we conjecture that, as in the setting of exponential polynomials over $\Qbar$, this problem happens only with simple functions, defined as follows.

\begin{defi}
 A {\em simple} element of $\calE$ is an $E$-function of the form
 \begin{equation}
 \label{eqsimple}
 g(x)=x^{-\omega} e^{\alpha x} \sum_{i=1}^N \lambda_i e^{\beta r_i x}
 \end{equation}
with $\alpha,\beta\in\Qbar$, $\beta\neq0$, $N\geq 2$, $\lambda_1,\ldots,\lambda_N\in\Qbar\etoile$, and pairwise distinct rational numbers $r_1,\ldots,r_N$; the integer $\omega$ is chosen so that $f$ is holomorphic at $0$ and $g(0)\neq0$. The {\em support} of $g$, denoted by $\supp(g)$, is the $1$-dimensional vector space
 $${\rm Span}_{\Q}(\beta) = \{\beta r, \, r\in\Q\}\subset \Qbar$$
 which depends only on $g$ and not on the choice of $\alpha,\beta$, $\lambda_1,\ldots,\lambda_N$, $r_1,\ldots,r_N$.
\end{defi}

We point out that this definition is slightly different from the usual one in the setting of exponential polynomials over $\Qbar$ (which amounts to taking $\omega=0$ in Eq. \eqref{eqsimple}). For instance $\frac{\sin(x)}{x}$ (which is an $E$-function but not an exponential polynomial) is simple, whereas $\sin(x)$ is not, with our definition. This modification is necessary for Conjecture~\ref{conjfact} below to be reasonable (otherwise, conjecturally $\frac{\sin(x)}{x}$ would have no factorization).

An  $E$-function is simple if, and only if, it is not a unit and it can be written as 
$$g(x) = x^{-\omega} e^{\alpha x} \sum_{i=1}^N \lambda_i e^{\beta_i x} $$
with $N\geq 2$, $\alpha, \lambda_1,\ldots,\lambda_N\in\Qbar$, $\beta_1,\ldots,\beta_N\in\Qbar$ such that $\beta_i/\beta_j\in\Q$ for any $i,j$ such that $\beta_j\neq 0$, and $\omega\in\N$ such that $g(0)\neq 0$. With this notation, the support of $g$ is $\Span_\Q(\beta_i)$, for any $i$ such that $\beta_i\neq 0$. 

We shall use the following characterization very often; here $\oun P$ is the order of multiplicity of $1$ as a root of $P$.

\begin{lem}
 \label{lemsplebis} A function $g$ is a simple $E$-function if, and only if, it can be written as 
 \begin{equation}
 \label{eqsple2} g(x) = x^{-\oun P} u(x) P(e^{\beta x}) 
 \end{equation}
with $\beta\in\Qbar\etoile$, a unit $u\in\Ecroix$, and $P\in\Qbar[X]$ such that $P(0)\neq 0$. Moreover, we have $\supp(g)=\Span_\Q(\beta)$.
\end{lem}
We point out that this expression is not unique: for instance $e^{\beta x}=(e^{\beta x/N})^N$ so that $\beta$ and $P(X)$ can be replaced with $\beta/N$ and $P(X^N)$, for any $N\geq 1$. Given finitely many simple functions with the same support, this remark enables one to write them as \eqref{eqsple2} with the same $\beta$. Indeed given finitely many elements $\beta_i \in V\setminus\{0\}$, where $V$ is a $1$-dimensional vector space over $\Q$, there exists $\beta \in V\setminus\{0\}$ such that all $\beta_i$ are integer multiples of $\beta$ (see the proof of Corollary~\ref{cordicho} in \S \ref{secpolexp}). 

\bigskip

\begin{proof}[Proof of Lemma~\ref{lemsplebis}] Let $g$ be a simple function, written as \eqref{eqsimple}. Modifying $\alpha$ if necessary, we may assume that $r_i\geq 0$ for any $i$, with equality for an index $i$. Let $D$ be a common denominator of the rationals $r_i$; replacing $\beta$ with $\beta/D$ we may assume that $D=1$, so that $r_i\in\N$ for any $i$. Then letting $u(x)=e^{\alpha x}$ and $P(X)=\sum_{i=1}^N \lambda_i X^{r_i}\in\Qbar[X]$, we have $P(0)\neq 0$ (because $\lambda_i\neq 0$ for all $i$, and $r_i=0$ for some $i$), and $g(x) = x^{-\omega} u(x) P(e^{\beta x})$. We may write $P(X) = (X-1)^{\oun P}Q(X)$ with $Q(X)\in\Qbar[X]$ such that $Q(1)\neq 0$. Then we have $g(x) = x^{-\omega} u(x) (e^{\beta x}-1)^{\oun P} Q(e^{\beta x})$, and the function $Q(e^{\beta x})$ does not vanish at $x=0$. Now both $g(x)$ and $\frac{e^{\beta x}-1}{x}$ are holomorphic and do not vanish at $0$, so that $\omega=\oun P$ and $g$ is of the form \eqref{eqsple2}. The converse can be proved easily along the same lines. This concludes the proof of Lemma~\ref{lemsplebis}. 
\end{proof}

\begin{defi}
 A non-zero element $f\in\calE$ is said to be {\em normalized} if its Taylor expansion $\sum_{n=0}^\infty a_n x^n$ around the origin satisfies 
 $$a_0=\ldots=a_{p-1}=0, \quad a_p=1, \quad a_{p+1}=0$$
 for some $p\geq 0$.   
\end{defi}
In other words, $f$ is normalized if, and only if,  its first non-zero Taylor coefficient at the origin is equal to $1$, and the next one vanishes.
The point is that for any non-zero $f\in\calE$, there exists a unique $u\in\Ecroix$ such that $fu$ is normalized. Moreover, any product of normalized $E$-functions is again normalized. This allows one to have really equalities between functions, not only ``up to units''. We conjecture that the following analogue of Ritt's factorization theorem holds. 

\begin{conj}
 \label{conjfact} Let $f$ be a non-zero $E$-function. Then there exist a unit $u\in\Ecroix$, simple normalized $E$-functions $s_1,\ldots,s_p$ with pairwise distinct supports (where $p\geq 0$), and irreducible normalized $E$-functions $h_1,\ldots,h_n $ (with $n\geq 0$), such that 
 \begin{equation}
 \label{eqdcp1}
 f=u s_1\ldots s_p h_1\ldots h_n.
 \end{equation}
 Moreover $u, s_1,\ldots ,s_p, h_1,\ldots, h_n$ are unique.
\end{conj}

To state Conjecture~\ref{conjfact} in a different way, we denote by $\calV$ the set of all $\Q$-vector spaces of dimension 1 contained in $\Qbar$, and by $\calI$ the set of all normalized irreducible $E$-functions.

\begin{prop}
 \label{propequivdcp} Conjecture~\ref{conjfact} is equivalent to the following statement. 

 Let $f$ be a non-zero $E$-function. There exist a unit $u$, a function $s_V$ which is either equal to 1 or simple normalized with support $ V$ (for each $V\in\calV$), and a non-negative integer $n_h$ (for each $h\in\calI$), such that 
 \begin{equation}
 \label{eqdcp2} f=u\Big( \prod_{V\in \calV} s_V\Big)\Big( \prod_{h\in \calI} h^{n_h}\Big),
 \end{equation}
 with $s_V=1$ for all but finitely many $V\in\calV$, and $n_h=0$ for all but finitely many $h\in\calI$. 

 Moreover $u$, $(s_V)_{V\in\calV}$ and $(n_h)_{h\in\calI}$ are uniquely determined by $f$.
 \end{prop}

We notice that in the products of Eq. \eqref{eqdcp2}, all factors are equal to the constant function $1$ except finitely many of them.

The proof of Proposition~\ref{propequivdcp}
 is straightforward since decompositions \eqref{eqdcp1} and \eqref{eqdcp2} are equivalent. Indeed for $h\in\calI$, $n_h$ is the number of $i$ such that $h_i=h$; and for $V\in\calV$, $s_V=s_i$ if there is a (necessarily unique) index $i$ such that $\supp( s_i)=V$, and $s_V=1$ otherwise.

 \bigskip

 {\bf Until the end of \S \ref{secstruct}, we assume that Conjecture~\ref{conjfact} holds.}
 
 \begin{defi}
 Let $f$ be a non-zero $E$-function, and $h\in\calI$. The {\em $h$-adic valuation of $f$}, denoted by $v_h(f)$, is the exponent $n_h$ of $h$ in the decomposition \eqref{eqdcp2} of $f$. \end{defi}

Let $x_0\in\Qbar$. Then $x-x_0$ is an irreducible $E$-function. Indeed, if $x-x_0 = f_1f_2$ with $f_1,f_2\in\calE$, then up to swapping $f_1$ and $f_2$ we may assume that $f_1(x_0)=0$; then \cite[Proposition~4.1]{beukers} yields $f_3\in\calE$ such that $f_1(x)=(x-x_0)f_3(x)$, and therefore $f_2f_3=1$ so that $f_2$ is a unit. Now consider the $E$-function $h_{x_0}$ defined by
$h_0(x)=x$ and, if $x_0\neq 0$, $h_{x_0}(x)=\frac{-1}{x_0}e^{x/x_0}(x-x_0)=\Big(1-\frac{x}{x_0}\Big)e^{x/x_0}$. Then $h_{x_0}$ is irreducible too, and it is normalized so that $h_{x_0}\in\calI$.

\begin{prop}
 \label{propordreannul} Let $x_0\in\Qbar$, and $f$ be a non-zero $E$-function. Then $v_{h_{x_0}}(f)$ is the order of vanishing of $f$ at $x_0$, where $h_0(x)=x$ and, if $x_0\neq 0$, 
 $$h_{x_0}(x)= \Big(1-\frac{x}{x_0}\Big)e^{x/x_0}.$$
\end{prop}

We point out that Proposition~\ref{propordreannul} applies to any $x_0\in\Qbar$, including $x_0=0$, whereas the corresponding statement with exponential polynomials would be false for $x_0=0$: indeed $e^x-1$ vanishes at $0$ but is not divisible by 
$x$ in the ring of exponential polynomials.

\bigskip

\begin{proof}[Proof of Proposition~\ref{propordreannul}] Let $x_0$ be an algebraic number. First of all, we claim that if $g$ is simple then $g(x_0)\neq 0$. This is part of the definition of a simple function if $x_0=0$. Otherwise, Lemma~\ref{lemsplebis} yields $g(x)=x^{-\oun P}u(x) P(e^{\beta x})$ for some non-zero $\beta\in\supp(g)\subset\Qbar$, $u\in\Ecroix$ and $P\in\Qbar[X]$. Now $e^{\beta x_0}$ is transcendental due to the Hermite-Lindemann Theorem, and $P\neq 0$, so that $P(e^{\beta x_0})$ is non-zero and finally $g(x_0)\neq 0$.

Now let us prove that if $h\in\calI$ is such that $h(x_0)=0$, then $h=h_{x_0}$. Indeed, using \cite[Proposition~4.1]{beukers}, $h$ can be written as $(x-x_0)h_1(x)$ with $h_1\in\calE$. Since $x-x_0$ vanishes at $x_0$, it is not a unit. Now $h$ is irreducible, so that $h_1$ is a unit: we have $h(x)=\lambda e^{\alpha x} (x-x_0)$ for some $\lambda,\alpha\in\Qbar$ with $\lambda\neq0$. To conclude we recall that $h$ is normalized. If $x_0=0$ then $h(0)=0$ and $h'(0)=\lambda$ so that $\lambda=1$ and $\alpha=0$: we have $h(x)=x=h_0(x)$. On the contrary, if $x_0\neq 0$ then $h(0)=-\lambda x_0$ so that $\lambda=-1/x_0$ and $\alpha=1/x_0$. This concludes the proof that if $h\in\calI$ is such that $h(x_0)=0$, then $h=h_{x_0}$. 

To conclude the proof of Proposition~\ref{propordreannul}, we notice that in Eq. \eqref{eqdcp2} no factor on the right-hand side vanishes at $x_0$, except the one that corresponds to $h_{x_0}$.
\end{proof}
\bigskip

Recall that {\em $g$ divides $f$} (with $f,g\in\calE$) means that $f=gg_1$ for some $g_1\in\calE$; Jossen's conjecture asserts that this is equivalent to $f/g$ being entire. We shall explain now how to translate the property that $g$ divides $f$ in terms of the decompositions \eqref{eqdcp2} of $f$ and $g$.

Proposition~\ref{propordreannul} shows that $h_{x_0}$ divides a non-zero $f\in\calE$ if, and only if, $f(x_0)=0$. Indeed a normalized irreducible $E$-function $h$ divides $f\in\calE\setminus\{0\}$ if, and only if, $v_h(f)\geq 1$: this follows at once from the unique decomposition \eqref{eqdcp2}.

Given simple functions $g_1, g_2$ with the same support $V\in\calV$, using Lemma~\ref{lemsplebis} and the remark following it, we write 
$$ g_1(x) = x^{-\oun P_1} u_1(x) P_1(e^{\beta x}) \quad \mbox{and} \quad 
 g_2(x) = x^{-\oun P_2} u_2(x) P_2(e^{\beta x})$$
with $\beta\in V \setminus\{0\}\subset \Qbar\etoile$, units $u_1,u_2\in\Ecroix$, and $P_1,P_2\in\Qbar[X]$ such that $P_1(0)\neq 0$ and $P_2(0)\neq 0$. We claim that $g_1$ divides $g_2$ in $\calE$ if, and only if, $P_1$ divides $P_2$ in $\Qbar[X]$. If $P_1$ divides $P_2$ this is clear. Otherwise, upon dividing $P_1$ and $P_2$ by their gcd we may assume they are coprime, and that $P_1$ is non constant. Then $P_1$ has a root $y_1\in\C$, which is not a root of $P_2$. Since $P_1(0)\neq 0$ we have $y_1\neq 0$: there exists $x_1\in\C$ such that $e^{\beta x_1}=y_1$. If $x_1=0$ then we may also choose $x_1=2i\pi/\beta$, so we assume $x_1\neq 0$. Then $g_1(x_1)=0$ and $g_2(x_1)\neq0$, so that $g_1$ does not divide $g_2$ in $\calE$ (recall that $E$-functions are entire, and therefore holomorphic at $x_1$). This enables us to understand divisibility amongst simple functions with the same support. To understand divisibility in $\calE$ we need the following definition.

\begin{defi} \label{defisimplepart}
 Let $f$ be a non-zero $E$-function, and $V\in \calV$. The {\em simple part with support $V$} of $f$, denoted by $s_V(f)$, is the function $s_V$ in the decomposition \eqref{eqdcp2} of $f$. 
\end{defi}

Notice there is a slight abuse in this terminology: if $s_V(f)=1$, it is not simple and therefore has no support. We call it the simple part with support $V$ of $f$ anyway.

With this definition we have the following result.

 \begin{prop}
 \label{propdiviE} Let $f_1,f_2$ be non-zero $E$-functions. Then $f_1$ divides $f_2$ in $\calE$ if, and only if, the following properties hold:
 \begin{itemize}
 \item For any $V\in\calV$, $s_V(f_1)$ divides $s_V(f_2)$.
 \item For any $h\in\calI$, $v_h(f_1)\leq v_h(f_2)$.
 \end{itemize} 
 \end{prop}

\begin{proof}[Proof of Proposition~\ref{propdiviE}] If $f_1$ divides $f_2$, we have $f_2=f_1f$ for some $f\in\calE\setminus\{0\}$. Comparing the decompositions \eqref{eqdcp2} of $f$, $f_1$ and $f_2$ gives $s_V(f_2)=s_V(f_1)s_V(f)$ and $v_h(f_2)=v_h(f_1)+v_h(f)$ by unicity, and this concludes the proof. Conversely, let $V\in\calV$ and assume that $s_V(f_1)$ divides $s_V(f_2)$. Then we have $s_V(f_2) = s_V(f_1) f_V$ for some $f_V\in \calE\setminus\{0\}$. Decomposing $f_V $ as in \eqref{eqdcp2} yields a decomposition of $s_V(f_2)$. By unicity of the latter, we have $s_V(f_2) = s_V(f_1)s_V( f_V)$. Moreover, if $s_V(f_2)=s_V(f_1)=1$ then $f_V=1$ so that $s_V(f_V)=1$. Finally, if we assume that for any $V\in\calV$, $s_V(f_1)$ divides $s_V(f_2)$ and 
for any $h\in\calI$, $v_h(f_1)\leq v_h(f_2)$, then letting 
$$ f =\Big( \prod_{V\in \calV} s_V(f_V)\Big)\Big( \prod_{h\in \calI} h^{v_h(f_2)-v_h(f_1)}\Big) $$
(where in each product, only finitely many factors are different from 1), we have $f_2=f_1f$. This concludes the proof of Proposition~\ref{propdiviE}.
\end{proof}

\bigskip

 Proposition~\ref{propdiviE} enables us to define gcd's in $\calE$ as follows, starting with simple functions with the same support.
\begin{defi}\label{defisimplegcd}
 Let $g_1$, $g_2$ be simple functions with the same support $V$. Using Lemma~\ref{lemsplebis} we may write $g_1(x)=x^{-\oun P_1} u_1(x) P_1(e^{\beta x})$ and $g_2(x)=x^{-\oun P_2} u_2(x)P_2(e^{\beta x})$ where $u_1,u_2\in\Ecroix$, $P_1,P_2\in\Qbar[X]$ don't vanish at 0, and $\beta\in V\setminus\{0\}\subset \Qbar\etoile$. The {\em gcd of $g_1$ and $g_2$}, denoted by $\gcd(g_1,g_2)$, is then $P(e^{\beta x})$ where $P$ is the gcd of $P_1$ and $P_2$ in $\Qbar[X]$. 
\end{defi}

The following lemma shows that this gcd is independent (up to a unit of $\calE$) of the choice of $\beta$, $P_1$, $P_2$, $u_1$, $u_2$. 

\begin{lem}
 \label{lemgcdsple} Let $g_1$, $g_2$, $g_3$ be simple functions with the same support $V$. Then $g_3$ divides $\gcd(g_1,g_2)$ if, and only if, $g_3$ divides both $g_1$ and $g_2$.
\end{lem}

The proof of this lemma is straightforward upon writing $g_i(x)=x^{-\oun P_i} u_i(x)P_i(e^{\beta x})$ where $\beta$ is independent of $i\in\{1,2,3\}$, and using the remark before Definition~\ref{defisimplepart}. 

We can now generalize the definition of gcd.

\begin{defi} \label{defgcd}
 Let $f_1,f_2$ be non-zero $E$-functions. The {\em gcd of $f_1$ and $f_2$}, denoted by $\gcd(f_1,f_2)$, is 
 $$ \Big( \prod_{V\in \calV} \gcd( s_V(f_1), s_V(f_2))\Big)\Big( \prod_{h\in \calI} h^{\min(v_h(f_1),v_h(f_2))}\Big) .$$
\end{defi}

This definition makes sense because of the following result, which is an immediate consequence of Proposition~\ref{propdiviE} and Lemma~\ref{lemgcdsple}. 

\begin{prop}
 \label{propgcd} Let $f_1,f_2,f_3$ be non-zero $E$-functions. Then $f_3$ divides $\gcd(f_1,f_2)$ if, and only if, $f_3$ divides both $f_1$ and $f_2$.
\end{prop}

To sum up the results obtained in this section, we state the following (see \cite[p. 4]{cg} for the definition and various other names of a gcd domain, including pseudo-Bezout ring \cite[p. 280]{bourbaki}).

\begin{theo}
 \label{thstruct} The ring $\calE$ is neither factorial nor noetherian. However, if Conjecture~\ref{conjfact} holds, it is a gcd domain.
\end{theo}

The non-factoriality of $\calE$ has been proved already using Eq. \eqref{eqnonfact}. To prove that $\calE $ is not noetherian, we show that $(I_N)_{N\geq 0}$ is an increasing sequence of ideals, where
$$I_N = \{f\in\calE, \quad \forall k \in\Z\quad f(k2^N\pi)=0\}.$$
Indeed, for any $N\geq 0$ we have $I_N\subset I_{N+1}$ and $f_N(x) = \sin(x/2^{N+1})$ belongs to $I_{N+1}\setminus I_{N}$.

At last, given $f_1,f_2\in\calE$, we conjecture that the ideal of $\calE$ generated by $f_1 $ and $f_2$ is not principal in general, so that there are no Bezout relations in $\calE$. 

\section{Zeros of $E$-functions}\label{seczerosE}

In this section we study zeros of $E$-functions. Our point of view is to state two conjectures dealing with  special $E$-functions, and then to deduce the general properties in Theorems \ref{th1} and \ref{th2}.

\begin{conj}
 \label{conjirred} Let $h_1$, $h_2$ be irreducible $E$-functions with (at least) a common zero in $\C$. Then $h_1=u h_2$ for some unit $E$-function $u$.
\end{conj}

This conjecture holds if $h_1(x)=x-x_0$ for some $x_0\in\Qbar$. Indeed, if $h_2(x_0)=0$ then $h_{x_0}$ divides $h_2$ (see Proposition~\ref{propordreannul} above). Now both $h_{x_0}$ and $h_2$ are irreducible, so that $h_2=u_1 h_{x_0} = u_2 h_1$ for some $u_1,u_2\in\Ecroix$. 

\begin{conj}
\label{conjzerolog} Let $\xi\in\C\etoile$ be such that $e^{\beta\xi}$ is algebraic for some $\beta\in\Qbar\etoile$. Then $h(\xi)\neq0$ for any irreducible $E$-function $h$.
\end{conj}
This result follows from the Hermite-Lindemann theorem if $h(x)=x-x_0$ for some $x_0\in\Qbar$. If $h$ is an exponential polynomial over $\Qbar$, the conclusion of Conjecture~\ref{conjzerolog} follows from Schanuel's conjecture (see Corollary~\ref{cordicho} in \S \ref{secpolexp}).

In particular Conjecture~\ref{conjzerolog} implies that no irreducible $E$-function vanishes at $\pi$, $\log 2$, etc.

\bigskip

We are now ready to  provide an analogue of Theorem~\ref{th10} for $E$-functions.

\begin{theo}
 \label{th1} Assume that Conjectures \ref{conjfact}, \ref{conjirred} and \ref{conjzerolog} hold. Let $f_1$, $f_2$ be non-zero $E$-functions. Then $\frac{f_1}{\gcd(f_1,f_2)}$ and $\frac{f_2}{\gcd(f_1,f_2)}$ have no common zero in $\C$. In other words, common zeros of $f_1$ and $f_2$ are exactly the zeros of $\gcd(f_1,f_2)$, and for any such $\xi$, the order of vanishing of $\gcd(f_1,f_2) $ at $\xi$ is the least of the orders of vanishing of $f_1$ and $f_2$ at~$\xi$. 

At last, this result holds unconditionally if $f_1$ and $f_2$ are simple.
 \end{theo}

Assuming that Conjecture~\ref{conjfact} holds ({\em i.e.}, that $E$-functions can be factored as in Ritt's theorem, so that gcd's exist), part $(ii)$ of Jossen's Conjecture~\ref{conj2intro} implies the conclusion of Theorem~\ref{th1}: coprime $E$-functions have no common zeros. In a converse way, we have the following.

\begin{coro}
 \label{cormultizeros} If Conjectures \ref{conjfact}, \ref{conjirred} and \ref{conjzerolog} hold then:
 \begin{itemize}
 \item Irreducible $E$-functions have only simple zeros.
 \item Conjecture~\ref{conj1intro} and Jossen's Conjecture~\ref{conj2intro}  hold.
 \end{itemize}
\end{coro}

\begin{proof}[Proof of Corollary~\ref{cormultizeros}] The first part can be proved exactly in the same way as Corollary~\ref{cormultizeros2}.

To deduce Conjecture~\ref{conj1intro}, we may therefore restrict (using the decomposition \eqref{eqdcp2}) to a simple function $g$, and even (using Lemma~\ref{lemsplebis}) to $g(x)=P(e^{\beta x})$ with $P\in\Qbar[X]$ and $\beta\in\Qbar\etoile$. Factoring $P$ in $\Qbar[X]$, what remains to consider is the case where $P(X)=X-y_0$ for some $y_0\in\Qbar$. Then $g'$ does not vanish, so that $g$ has only simple roots. 

To deduce Conjecture~\ref{conj2intro}, we consider $f,g\in\calE\setminus\{0\}$ such that $f/g$ is entire. Dividing $f$ and $g$ by their gcd if necessary, we may assume that $\gcd(f,g)=1$. Then all zeros of $g$ are zeros of $f$, but $f$ and $g$ have no common zeros due to Theorem~\ref{th1}. Therefore $g$ does not vanish: it is a unit of $\calE$, and $f/g$ is an $E$-function. This proves the first part of Jossen's conjecture; the second one follows at once from Theorem~\ref{th1}.
\end{proof}

\bigskip

\begin{proof}[Proof of Theorem~\ref{th1}] Upon dividing $f_1$ and $f_2$ by their gcd (which exists unconditionally if $f_1$ and $f_2$ are simple), we may assume that $\gcd(f_1,f_2)=1$. Let $\xi\in\C$ be a common zero of $f_1$ and $f_2$. Then in the decomposition \eqref{eqdcp2} of $f_1$ (resp. of $f_2$), at least one factor vanishes at $\xi$. There are 3 possibilities.

First, let us consider the case where  there exist $V_1,V_2\in\calV$ such that $g_1=s_{V_1}(f_1)$ and $g_2= s_{V_2}(f_2)$ vanish at $\xi$; in particular this happens if $f_1$ and $f_2$ are simple. Recall that $g_i(x)$ can be written as $x^{-\oun P_i} u_i(x)P_i(e^{\beta_i x})$ with $u_i\in\Ecroix$, $\beta_i\in\supp(g_i)\setminus\{0\}$, and $P_i\in\Qbar[X]\setminus\{0\}$. Since $g_i(\xi)=0$, we have $\xi\neq 0$ and  $e^{\beta_i \xi}$ is a root of $P_i$ and therefore an algebraic number. Now $e^{\beta_1 \xi}$ and $e^{\beta_2 \xi}$ are both algebraic, so that $\beta_1/\beta_2$ is rational  using the Gel'fond-Schneider Theorem. Accordingly we may write $g_i(x) = x^{-\oun P_i} u_i(x)P_i(e^{\beta x})$ with $\beta$ independent from $i$; and $g=\gcd(g_1,g_2)$ is defined by $g(x) = P(e^{\beta x})$ where $P$ is the gcd of $P_1$ and $P_2$ in $\Qbar[X]$. If $\xi$ is a common zero of $g_1/g$ and $g_2/g$ then $e^{\beta \xi}$ is a common zero of $P_1/P$ and $P_2/P$, which is impossible since these polynomials are coprime. 

If there exist $h_1,h_2\in\calI$ such that $v_{h_1}(f_1)$ and $v_{h_2}(f_2)$ are positive and $h_1(\xi)=h_2(\xi)=0$, then Conjecture~\ref{conjirred} implies $h_1=h_2$, and this irreducible $E$-function divides $\gcd(f_1,f_2)=1$: this is a contradiction.

The last possibility (up to swapping $f_1$ and $f_2$) is that there exist $V\in\calV$ and $h\in\calI$ such that both $s_V(f_1)$ and $h$ vanish at $\xi$, with $v_h(f_2)\geq 1$. Writing $s_V(f_1)(x)=x^{-\oun P} u(x)P(e^{\beta x})$ with $u\in\Ecroix$, $P\in\Qbar[X]$ and $\beta\in \Qbar$, we obtain that $e^{\beta \xi}$ is algebraic. Since $h(\xi)=0$ this contradicts Conjecture~\ref{conjzerolog}, and concludes the proof of Theorem~\ref{th1}.
\end{proof}

\bigskip

We can deduce now an analogue of Corollary~\ref{cordicho} for $E$-functions.

\begin{theo} \label{th2} Assume that Conjectures \ref{conjfact}, \ref{conjirred} and \ref{conjzerolog} hold. Let $\xi\in\C$ be a zero of a non-zero $E$-function. Then one, and only one, of the following holds:
 \begin{itemize}
 \item We have $\xi\neq0$, $e^{\beta\xi}\in\Qbar$ for some $\beta\in\Qbar\etoile$, and for any $f\in\calE$ with $f(\xi)=0$ there exists $N\geq 1$ such that $f$ is divisible by $e^{\beta x/N}-e^{\beta\xi/N}$ in $\calE$.
 \item We have $h(\xi)=0$ for some irreducible $h\in\calE$, and $h$ divides in $\calE$ any $E$-function that vanishes at $\xi$.
 \end{itemize}
\end{theo}

This result shows that there are (conjecturally) two types of zeros of $E$-functions. The first one corresponds to $\xi=\frac{\log(\alpha)}{\beta}$ with $\alpha,\beta\in\Qbar\etoile$ and any determination of $\log(\alpha)$; the second one includes (conjecturally) zeros of Bessel functions $J_\alpha$, $\alpha\in\Q\setminus\{\pm1/2\}$, and values at algebraic points of the Lambert $W$ function. Moreover Theorem~\ref{th2} gives a kind of generalization of minimal polynomials (resp. conjugates) of algebraic numbers; in the second case, it would be the irreducible function $h$, which is unique up to multiplication by a unit of $\calE$ (resp. its zeros).

If $\xi$ is algebraic then the conclusion of Theorem~\ref{th2} holds inconditionally. Indeed the first case cannot occur due to the Hermite-Lindemann Theorem, and the second one holds with $h(x)=x-\xi$ using \cite[Proposition~4.1]{beukers}. We see that the minimal function of $\xi$ would be $x-\xi$, which is reasonable: recall that $\Qbar[X]\subset \calE$. Maybe the minimal polynomial of $\xi$ over $\Q$ could be recovered in this setting by restricting to $E$-functions with rational coefficients, but we did not try to do it.

\bigskip

\begin{proof}[Proof of Theorem~\ref{th2}] First of all, Conjecture~\ref{conjzerolog} shows that both cases of Theorem~\ref{th2} cannot occur simultaneously. Let $f_0\in\calE\setminus\{0\}$ and $\xi\in\C$ be such that $f_0(\xi)=0$. Decomposing $f_0$ as in Conjecture~\ref{conjfact}, at least one factor vanishes at $\xi$: either a simple function or an irreducible one.

If a simple function $g(x)=x^{-\oun P}u(x)P(e^{\beta x})$ vanishes at $\xi$, then $e^{\beta \xi}$ is algebraic. We write $V={\rm Span}_{\Q}(\beta)=\supp(g)$ and consider any $f\in\calE\setminus\{0\}$ such that $f(\xi)=0$. Theorem~\ref{th1} shows that $\xi$ is a zero of $\gcd(f,g)=\gcd(s_V(f),g)$. We have $s_V(f)(x)=x^{-\oun Q}v(x)Q(e^{\beta' x})$ for some $v\in\Ecroix$, $Q\in\Qbar[X]$ and $\beta'\in V\setminus\{0\}$. Since $\beta$ and $\beta'$ span the same $\Q$-vector space, there exists $N\geq 1$ such that $\beta'$ is an integer multiple of $\beta/N$. Then $s_V(f)$, $g$ and $\gcd(s_V(f),g)$ can be written (up to units) as polynomials in $e^{\beta x/N}$ multiplied by suitable powers of $x$. This provides $S\in\Qbar[X]$ and $w\in\Ecroix$ such that $\gcd(s_V(f),g)(x)=x^{\oun S}w(x)S(e^{\beta x/N})$ vanishes at $\xi$. We have $S(e^{\beta \xi/N})=0$ so that $S(X) = (X - e^{\beta \xi/N})T(X)$ for some $T\in\Qbar[X]$, since $e^{\beta \xi/N}$ is algebraic. Then $e^{\beta x/N}-e^{\beta \xi/N}$ divides $\gcd(s_V(f),g)=\gcd(f,g)$, and therefore $f$, in $\calE$. This concludes the proof in the case where a simple function vanishes at $\xi$. 

Let us assume now that $h(\xi)=0$ for some irreducible $E$-function $h$; we assume $h$ to be normalized. For any $f\in\calE\setminus\{0\}$ such that $f(\xi)=0$, Theorem~\ref{th1} shows that $\gcd(f,h)=h^{\min(1, v_h(f))}$ vanishes at $\xi$. Therefore $v_h(f)\geq 1$ and $h$ divides $f$. This concludes the proof of Theorem~\ref{th2}.
\end{proof}

\noindent St\'ephane Fischler, Universit\'e Paris-Saclay, CNRS, Laboratoire de math\'ematiques d'Orsay, 91405 Orsay, France.

\medskip

\noindent Tanguy Rivoal, Universit\'e Grenoble Alpes, CNRS, Institut Fourier, CS 40700, 38058 Grenoble cedex 9, France.

\bigskip

\noindent Keywords: $E$-functions, Exponential polynomials, Schanuel's conjecture.

\bigskip

\noindent MSC 2020: 11J91, 33B20. 

\end{document}